\documentclass{amsart}

\usepackage{comment}

\usepackage[all]{xy}

\usepackage{graphicx}
\usepackage[colorlinks=true]{hyperref} 
\usepackage[usenames,dvipsnames]{xcolor}
\usepackage[centering,margin=1.1in]{geometry}

\newcommand{\qq}{\mathbb{Q}}   
\newcommand{\Z}{\mathbb{Z}}     
\newcommand{\HY}{\mathcal{SE}}

\newcommand{\casetwo}[4]{\left\{ \begin{array}{ll} #1 &\mbox{if $#2$} \\[2mm] #3 &\mbox{if $#4$}\,. \end{array} \right.}

\newcommand{\casetwoct}[3]{\left\{ \begin{array}{ll} #1 &\mbox{if $#2$} \\[2mm] #3 &\mbox{otherwise}\,, \end{array} \right.}

\theoremstyle{plain} 
\newtheorem{theo}{Theorem}[section]
\newtheorem{lem}[theo]{Lemma}
\newtheorem{prop}[theo]{Proposition}
\newtheorem{cor}[theo]{Corollary}
\newtheorem{conj}[theo]{Conjecture}

\theoremstyle{definition}
\newtheorem{example}[theo]{Example}
\newtheorem{defi}[theo]{Definition}

\theoremstyle{remark}
\newtheorem{rem}[theo]{Remark}
\newtheorem{rems}[theo]{Remarks}

\title[A Schubert basis in equivariant elliptic cohomology]
{A Schubert basis in equivariant elliptic cohomology}

\author[C.~Lenart]{Cristian Lenart}
\address[Cristian Lenart]{Department of Mathematics and Statistics, State University of New York at Albany, 
Albany, NY 12222, U.S.A.}
\email{clenart@albany.edu}
\urladdr{http://www.albany.edu/\~{}lenart/}

\author[K.~Zainoulline]{Kirill Zainoulline}
\address[Kirill Zainoulline]{Department of Mathematics and Statistics, University of Ottawa, 585 King Edward Street, Ottawa, ON, K1N 6N5, Canada}
\email{kirill@uottawa.ca}
\urladdr{http://mysite.science.uottawa.ca/kzaynull/}

\thanks{C.L. was partially supported by the NSF grant DMS--1362627. K.Z. was partially supported by the NSERC Discovery grant  385795-2010 and the Early Researcher Award (Ontario). \\
Keywords: Schubert calculus, elliptic cohomology, flag variety, Hecke algebra, Kazhdan-Lusztig basis \\
MSC: 14M15, 14F43, 55N20, 55N22, 19L47, 05E99}

\begin{document}

\begin{abstract}
 We address the problem of defining Schubert classes independently of a reduced word in equivariant elliptic cohomology, based on the Kazhdan-Lusztig basis of a corresponding Hecke algebra. We study some basic properties of these classes, and make two important conjectures about them: a positivity conjecture, and the agreement with the topologically defined Schubert classes in the smooth case. We prove some special cases of these conjectures.
\end{abstract}

\maketitle

\section{Introduction} Modern Schubert calculus has been mostly concerned with the equivariant singular cohomology and $K$-theory (as well as their quantum deformations) of generalized flag manifolds $G/B$, where $G$ is a connected complex semisimple Lie group and $B$ a Borel subgroup; Kac-Moody flag manifolds have also been studied, but we will restrict ourselves here to the finite case. The basic results in Schubert calculus for other oriented cohomology theories have only been obtained recently in \cite{CPZ, CZZ,CZZ1,CZZ2,hhhcge,HK,kakecf}. After this main theory has been developed, the next step is to give explicit formulas, thus generalizing well-known results in singular cohomology and $K$-theory, which are usually based on combinatorial structures. 

A first contribution in this direction is our paper \cite{LZ}. This focuses on an oriented cohomology corresponding to a singular cubic curve (in Weierstrass form), more precisely, to a singular elliptic formal group law (called hyperbolic), which we view as the first interesting case after $K$-theory in terms of complexity. (The correspondence between generalized cohomology theories and formal group laws is explained below.) The main result was concerned with extending the combinatorial formulas for localizations of Schubert classes in (torus equivariant) cohomology and $K$-theory, which are due to Billey \cite{Bi99} and Graham-Willems \cite{graekt,wilcke}, respectively. 

The main difficulty beyond $K$-theory is the fact that the topologically defined cohomology classes corresponding to a Schubert variety depend on its chosen Bott-Samelson desingularization; this is why they are called Bott-Samelson classes. Thus, these classes depend on a reduced word for the given Weyl group element, which is not the case in ordinary cohomology and $K$-theory, where we have naturally defined Schubert classes. In this paper we consider the problem of defining Schubert classes independently of a reduced word. We focus on the singular elliptic (hyperbolic) case only and we base our construction on the Kazhdan-Lusztig basis of a corresponding Hecke algebra, so we use the term Kazhdan-Lusztig (KL) Schubert classes. 

We chose to work in the equivariant setting, because of a positivity conjecture in this case, namely \cite[Conjecture~6.4]{LZ}, whereas such a property does not hold in the non-equivariant case. We conjecture the same positivity property for our KL-Schubert classes. Moreover, in the equivariant case we can use the GKM model for the corresponding cohomology of the flag variety, and we have an easy formula for the topologically defined Schubert classes in the smooth case. We conjecture that our KL-Schubert classes coincide with the latter in the smooth case. We prove some special cases of these conjectures.

\section{Background}\label{background} We briefly recall the main results in Schubert calculus for generalized cohomology theories.

\subsection{Complex oriented cohomology theories} A (one dimensional, commutative) {\em formal group law} over a commutative ring $R$ is a formal power series $F(x,y)=\sum_{i,j}a_{ij}x^iy^j$ in $R[[x,y]]$ satisfying (see \cite[p.4]{levmor-book})
\begin{equation} F(x,y)=F(y,x)\,,\;\;\;\;\;F(x,0)=x\,,\;\;\;\;\;F(x,F(y,z))=F(F(x,y),z)\,.\end{equation}
The {\em formal inverse} is the power series $\iota(x)$ in $R[[x]]$ defined by $F(x,\iota(x))=0$. 

Let $E^*(\cdot)$ be a complex oriented cohomology theory with base ring $R=E^*(\mbox{pt})$. By \cite{Qu71}, this is equipped with a formal group law $F(x,y)$ over $R$, which expresses the corresponding first Chern class, denoted $c(\cdot)$, of a tensor product of two line bundles ${\mathcal L}_1$ and ${\mathcal L}_2$ on a space $X$ in terms of $c({\mathcal L}_1)$ and $c({\mathcal L}_2)$: 
\begin{equation}\label{defF} c({\mathcal L}_1\otimes {\mathcal L}_2)=F(c({\mathcal L}_1),\,c({\mathcal L}_2))\,.\end{equation}

The reciprocal of Quillen's statement is false: there are formal group laws which do not come from  complex oriented cohomology theories. 
However, if one translates and extends the axiomatics of oriented theories 
into the algebraic context, which was done
by Levine and Morel in \cite{levmor-book}, then to any formal group law one can associate (by tensoring with algebraic cobordism over the Lazard ring) the respective \textit{algebraic} oriented cohomology theory. In this paper we will work in the more general algebraic setting, so $E^*(\cdot)$ stands for
the respective algebraic oriented cohomology.

We refer to the (finite type) generalized flag variety $G/B$, and we let $T$ be the corresponding maximal torus. We use freely the corresponding root system terminology. As usual, we denote the set of roots by $\Phi$, the subsets of positive and negative roots by $\Phi^+$ and $\Phi^-$, the simple roots and corresponding simple reflections by $\alpha_i$ and $s_i$ (for $i=1,\ldots,n$, where $n$ is the rank of the root system), the lattice of integral weights by $\Lambda$, the Weyl group by $W$, its longest element by $w_\circ$, and its strong Bruhat order by $\le$. For each $w\in W$, we have the corresponding Schubert variety $X(w):=\overline{BwB/B}$. Given an arbitrary weight $\lambda$, let $\mathcal{L}_\lambda$ be the corresponding line bundle over $G/B$, that is, $\mathcal{L}_\lambda:= G\times_B {\mathbb C}_{-\lambda}$, where $B$ acts on $G$ by right multiplication, and the
$B$-action on ${\mathbb C}_{-\lambda}={\mathbb C}$ corresponds to the character determined by $-\lambda$. 

We now consider the respective $T$-equivariant cohomology $E_T^*(\cdot)$ of spaces with a $T$-action, see e.g. \cite{hhhcge}. Its base ring $E_T^*(\mbox{pt})$ can be identified (after completion) with the {\em formal group algebra} 
\begin{equation}\label{defS}
S:=R[[y_\lambda]]_{\lambda\in \Lambda}/(y_0,\;
y_{\lambda+\nu}-F(y_{\lambda},y_{\nu}))
\end{equation}
of \cite[Def.~2.4]{CPZ}; in other words, $y_\lambda$ is identified with the corresponding first Chern class of ${\mathcal L}_\lambda$, cf. \eqref{defF}. The Weyl group acts on $S$ by $w(y_\lambda):=y_{w\lambda}$. 

The universal formal group corresponds to complex cobordism. In this paper we focus on cohomology, denoted $\HY^*(\cdot)$, whose corresponding formal group law is that of a singular cubic curve in Weierstrass form.
More precisely, such a curve
is given by $y^2+\mu_1 xy=x^3+\mu_2x^2$ \cite[Ch.III, App.A]{Si}.  (Observe that such a curve is always rational and has either a cusp or a node singularity.) The respective formal group law, called hyperbolic, is  (see \cite[Ch.IV]{Si}, \cite[Ex. 63]{BB10} and  \cite[Cor. 2.8]{BB11}
\begin{equation}\label{defell}
F_h(x,y)=\frac{x+y-\mu_1 xy}{1+\mu_2xy}\,,\qquad\mbox{defined over $R=\Z[\mu_1,\mu_2]$}\,.
\end{equation}
We let $u:=a_{12}=-\mu_2$. 

The hyperbolic formal group law is, in fact, a very natural one from a topological perspective. 
Via the one-to-one correspondence between genera and formal group laws it corresponds to a famous 2-parameter Hirzebruch genus $T_{x,y}$ (see \cite{Kr74}, where $x-y=\mu_1$ and $xy=\mu_2$). Observe that the genus $T_y=T_{1,y}$ appears in Hirzebruch's celebrated book \cite{Hir} 
by the name ``virtual generalized Todd genus''.  

Let us also mention some important special cases. 
In the trivial case $\mu_1=\mu_2=0$, $F_h(x,y)=F_a(x,y)$ is the additive group law, which corresponds to ordinary singular cohomology (or Chow group) $H^*(\cdot)$. If $\mu_2=0$, and $\mu_1=1$, resp. $\mu_1$ is not invertible, then $F_h(x,y)=F_m(x,y)$ is the corresponding version of the multiplicative formal group law, associated to $K$-theory, and connective $K$-theory, respectively.

\subsection{Schubert and Bott-Samelson classes}\label{strucetgb} Given a Weyl group element $w$, consider a reduced word ${I_w}=(i_1,\ldots,i_l)$ for it, so $w=s_{i_1}\ldots s_{i_l}$. There is a {\em Bott-Samelson resolution} of the corresponding Schubert variety $X(w)$, which we denote by $\gamma_{I_w}:\Gamma_{I_w}\rightarrow X(w)\hookrightarrow G/B$. This determines a so-called {\em Bott-Samelson class} in  $E^*_T(G/B)$ via the corresponding pushforward map, namely $(\gamma_{I_w})_{!}(1)$. Here we let 
\begin{equation}\label{zetaw}\zeta_{I_w}:=(\gamma_{I_w^{-1}})_{!}(1)\,,\end{equation}
 where $I_w^{-1}:=(i_l,\ldots,i_1)$ is a reduced word for $w^{-1}$; we use $I_{w^{-1}}$, rather than $I_w$, for technical reasons. Note that $\zeta_{\emptyset}$ is the class of a point (where $\emptyset$ denotes the reduced word for the identity).

By \cite[Thm.~3.7]{BE} the Bott-Samelson classes are independent of the corresponding reduced words only for cohomology and $K$-theories (we can say that connective $K$-theory is the ``last'' case when this happens). In these cases, the Bott-Samelson classes are the {\em Schubert classes}, and they form bases of $H_T^*(G/B)$ and $K_T(G/B)$ over the corresponding formal group algebra $S$, as $w$ ranges over $W$. (More precisely, the Schubert classes are the Poincar\'e duals to the fundamental classes of Schubert varieties in homology, whereas in $K$-theory they are the classes of {\em structure sheaves} of Schubert varieties.) More generally, an important result in generalized cohomology Schubert calculus says that, by fixing a reduced word $I_w$ for each $w$, the corresponding Bott-Samelson classes $\{\zeta_{I_w}\,:\,w\in W\}$ form an $S$-basis of $E_T^*(G/B)$. 

So we have the following diagram of oriented theories, respective formal group laws and specialization maps.
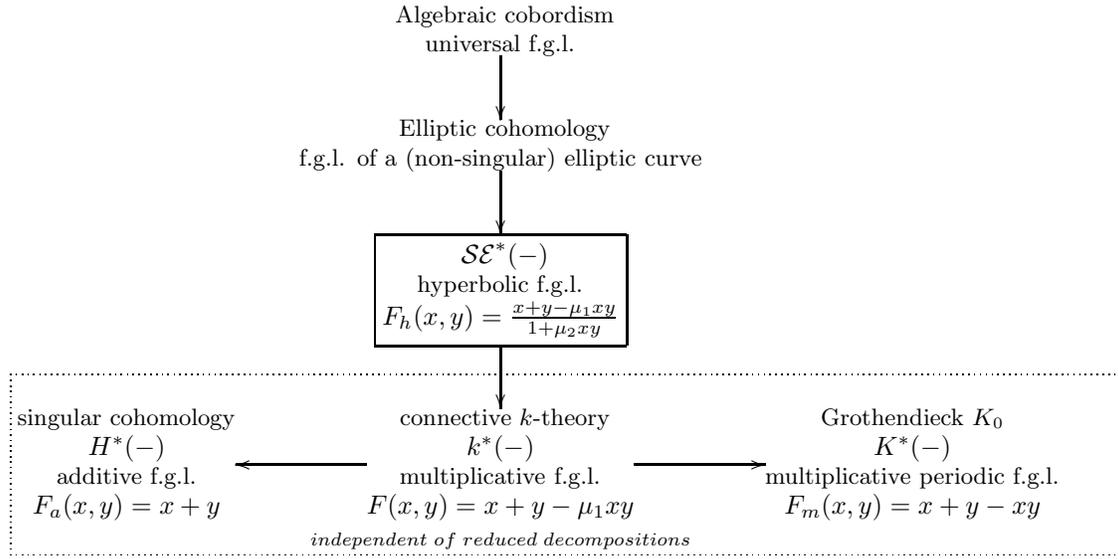
\[
\xymatrix{
 & *\txt{ {\small Algebraic cobordism}  \\ {\small universal f.g.l.}} \ar[d] & \\
 & *\txt{ {\small Elliptic cohomology}  \\ {\small f.g.l. of a (non-singular) elliptic curve}} \ar[d] & \\
 & *+[F]\txt{ $\HY^*(-)$ \\ {\small hyperbolic f.g.l.} \\ $F_h(x,y)=\tfrac{x+y-\mu_1 xy}{1+\mu_2xy}$} \ar[d] & \\
*\txt{{\small singular cohomology} \\ $H^*(-)$ \\ {\small additive f.g.l.} \\ $F_a(x,y)=x+y$ } & *\txt{ {\small connective $k$-theory} \\ $k^*(-)$ \\ {\small multiplicative f.g.l.} \\ $F(x,y)=x+y-\mu_1 xy$}  \ar[l] \ar[r] \ar@{}[d]|(.6){independent\; of \; reduced \; decompositions} & *\txt{{\small Grothendieck $K_0$} \\ $K^*(-)$  \\ {\small multiplicative periodic f.g.l.} \\ $F_m(x,y)=x+y-xy$} \\
 & &
\save "4,1" . "4,3"*+<6ex>[F.]\frm{} \restore
}
\]

There is a well-known model for $E_T^*(G/B)$ known as the {\em Borel model}, which we now describe. We start by considering the invariant ring 
\[S^W:=\{f\in S\,:\, wf=f \mbox{ for all $w\in W$}\}\,.\]
We then consider the coinvariant ring
\[S\otimes_{S^W} S:=\frac{S\otimes_R S}{\langle f\otimes 1-1\otimes f\,:\, f\in S^W\rangle}\,.\]
Here the product on $S\otimes_{S^W} S$ is given by $(f_1\otimes g_1)(f_2\otimes g_2):=f_1f_2\otimes g_1 g_2$. To more easily keep track of the left and right tensor factors, we set $x_\lambda:=1\otimes y_\lambda$ and $y_\lambda:=y_\lambda\otimes 1$. We use this convention whenever we work with a tensor product of two copies of $S$; by contrast, when there is a single copy of $S$ in sight, we let $x_\lambda=y_\lambda$. 
 
We are now ready to state a second important result in generalized cohomology Schubert calculus, namely that $S\otimes_{S^W} S$ is a rational model for $E_T^*(G/B)$, as an $S$-module; here the action of $y_\lambda\in S$ is on the left tensor factor, as the above notation suggests. 
Observe that, in general, $E_T^*(G/B)$ and  $S\otimes_{S^W} S$ are not isomorphic integrally (see \cite[Thm.~11.4]{CZZ}).

\subsection{The formal Demazure algebra}\label{fda} Following \cite[\S6]{HMSZ}, \cite[\S5]{CZZ} and \cite[\S3]{CZZ1}, consider the
localization $Q$ of $S$ along all $x_\alpha$, for $\alpha\in \Phi$ (note the change of notation, from $y_\lambda$ to $x_\lambda$, cf. the above convention),
and define the {\em twisted group algebra} $Q_W$  to be the smash product
$Q\# R[W]$, see \cite[Def.~6.1]{HMSZ}. More precisely, as an $R$-module, $Q_W$ is $Q\otimes_R R[W]$, while the
multiplication is given by
\begin{equation}\label{commrel}
q\delta_w \cdot q'\delta_{w'}=q(wq')\delta_{ww'},\qquad q,q'\in Q,\;\;
w,w'\in W.
\end{equation}

For simplicity, we denote
$\delta_i:=\delta_{s_i}$, $x_{\pm i}:=x_{\pm\alpha_i}$, and $x_{\pm i\pm j}:=x_{\pm\alpha_i\pm\alpha_j}$, for $i,j\in\{1,\ldots,n\}$; similarly for the $y$ variables.
Following \cite[Def.~6.2]{HMSZ} and \cite[\S3]{CZZ1}, for each $i=1,\ldots, n$,
we define in $Q_W$
\begin{equation}\label{defxy}
X_i:=\frac{1}{x_i}\delta_i-\frac{1}{x_i}=\frac{1}{x_i}(\delta_i-1),\qquad
Y_i:=X_i+\kappa_i=\frac{1}{x_{-i}}+\frac{1}{x_i}\delta_i=(1+\delta_i)\frac{1}{x_{-i}}\,,
\end{equation}
where $\kappa_i:=\frac{1}{x_{-i}}+\frac{1}{x_i}$. We call $X_i$ and $Y_i$ the {\em Demazure} and the {\em push-pull element}, respectively. 
The $R$-algebra ${\mathcal D}_F$ generated by multiplication with elements of $S$ and the
elements $\{X_i\}$, or 
$\{Y_i\}$, is called the {\em formal affine Demazure algebra}. Observe that its dual ${\mathcal D}_F^\star$ serves as an integral model for $E_T^*(G/B)$.

The algebras $Q_W$ and ${\mathcal D}_F$ act on $S\otimes_R Q$ by
\begin{equation}\label{action}h(f\otimes g)=f\otimes hg\;\;\;\;\;\mbox{and}\;\;\;\;\;\delta_w(f\otimes g):=f\otimes wg \,,\end{equation}
where $f\in S$, $g,h\in Q$, and $w\in W$. In fact, the Demazure and push-pull elements act on $S\otimes_R S$ and on the Borel model $S\otimes_{S^W} S$; we will focus on the latter action.


\begin{rem}\label{symmi} It is clear from definition that $Y_i$ is $S^{\langle s_i\rangle}$-linear. Moreover, if $\delta_i(f)=f$, i.e., $f$ is $s_i$-invariant, then $Y_i\,f=\kappa_i f$. It is easy to show (see, e.g., \eqref{fglcalc} below) that, for the hyperbolic formal group law, we have $\kappa_i=\mu_1$.
\end{rem}

Relations in the algebra ${\mathcal D}_F$ were given in \cite[Thm.~6.14]{HMSZ} and \cite[Prop.~8.10]{HMSZ}. In particular, in the case of the hyperbolic formal group law, we have the relations below:
\begin{itemize}
\item[(a)] For all $i$, we have (cf. Remark \ref{symmi})
\begin{equation}\label{square}
Y_i^2=\mu_1 Y_i\,.\end{equation}  
\item[(b)] If $\langle\alpha_i,\alpha_j^\vee\rangle=0$, so that $m_{ij}=2$, where $m_{ij}$ be the order of $s_is_j$ in $W$, then 
\begin{equation}\label{braid2}Y_iY_j=Y_jY_i\,.
\end{equation}
\item[(c)] If $\langle\alpha_i,\alpha_j^\vee\rangle=\langle\alpha_j,\alpha_i^\vee\rangle=-1$, so that $m_{ij}=3$ (i.e., in type $A_2$), we have 
\begin{equation}\label{braid3}Y_iY_jY_i-Y_jY_iY_j=\mu_2(Y_j-Y_i)=u(Y_i-Y_j)\,.\end{equation}
\item[(d)] If $m_{ij}=4$ (i.e., in type $B_2$), we have
\begin{equation}\label{braid4}Y_iY_jY_iY_j-Y_jY_iY_jY_i=2u(Y_iY_j-Y_jY_i)\,.\end{equation}
\item[(e)] If $m_{ij}=6$ (i.e., in type $G_2$), we have
\begin{equation}\label{braid6}Y_iY_jY_iY_jY_iY_j-Y_jY_iY_jY_iY_jY_i=4u(Y_iY_jY_iY_j-Y_jY_iY_jY_i)-3u^2(Y_iY_j-Y_jY_i)\,.\end{equation}
\end{itemize}
The relations in (c), (d), (e) are called {\em twisted braid relations}; they were derived in \cite[Example~4.12]{LNZ} from the general relations in \cite[Prop.~6.8]{HMSZ}, cf. also Example~\ref{b2g2}. In the case of ordinary cohomology and $K$-theory,  the twisted braid relations are the usual braid relations (since $\mu_2=0$).

Given a reduced word $I_w=(i_1,\ldots,i_l)$ for $w\in W$, define $X_{I_w}:=X_{i_1}\ldots X_{i_l}$ and $Y_{I_w}:=Y_{i_1}\ldots Y_{i_l}$. 
By \cite[Cor.~3.4]{CZZ1}, if we fix a
reduced word $I_w$ for each $w\in W$, then $\{X_{I_w}\,:\,w\in W\}$ and $\{Y_{I_w}\,:\,w\in W\}$ are bases of the free left
$Q$-module $Q_W$. Note that, in cohomology and $K$-theory, $X_{I_w}$ and $Y_{I_w}$ do not depend on the choice of the reduced word $I_w$, so we can simply write $X_w$ and $Y_w$. In fact, we will use the latter notation whenever we are in a similar situation.

A fundamental result in generalized cohomology Schubert calculus states that the Bott-Samelson classes $\zeta_{I_w}$, for ${I_w}=(i_1,\ldots,i_l)$, can be calculated recursively as follows (via the usual action of ${\mathcal D}_F$ on the {\em Borel model} of $E_T^*(G/B)$ in \eqref{action}):
\begin{equation}\label{pushpull}
\zeta_{I_w}=Y_{i_l}\ldots Y_{i_1}\,\zeta_{\emptyset}\,.
\end{equation}

\subsection{The GKM model of equivariant cohomology}\label{gkmsetup}
In the GKM model, see e.g. \cite{gkmeck,hhhcge,CZZ2}, we embed $E_T^*(G/B)$ into $\bigoplus_{w\in W}S$, with pointwise multiplication. 
This comes from the embedding 
\begin{equation}\label{gkmemb}i^*\,:\,E_T^*(G/B)\rightarrow\bigoplus_{w\in W}E_T^*({\rm{pt}})\simeq\bigoplus_{w\in W}S\,,\end{equation}
 where 
\begin{equation}\label{defi}i^*:=\bigoplus_{w\in W} i_w^*\,,\;\;\;\;\mbox{and}\;\;\;\;i_w\,:\,\mbox{pt}\rightarrow G/B\,,\:\mbox{ with $\mbox{pt}\mapsto w^{-1}$}\,.\end{equation}
There is a characterization of the image of this embedding. We denote the elements of $\bigoplus_{w\in W}S$ by $(f_w)_{w\in W}$; alternatively, we view them as functions $f:W\to S$. 

Using the Borel model for $E_T^*(G/B)$, we can realize the  GKM map $i^*$ in \eqref{gkmemb} as an embedding of $S\otimes_{S^W} S$ into $\bigoplus_{w\in W} S$. This map can be made explicit as
\begin{equation}\label{inc2}f\otimes g\mapsto (f\cdot(w g))_{w\in W}\,.\end{equation}
Via this map, the action \eqref{action} of the algebras $Q_W$ and ${\mathcal D}_F$ is translated as follows in the GKM model:
\begin{equation}\label{actiongkm} x_\lambda\cdot 1=(y_{w\lambda})_{w\in W}\,,\qquad \delta_v\,(f_w)_{w\in W}=(f_{wv})_{w\in W}\,.\end{equation}

As now the action of the push-pull operators $Y_i$ is made explicit, we can use \eqref{pushpull} to compute recursively the Bott-Samelson classes $\zeta_{I_w}$ in the GKM model, once we know the class $\zeta_{\emptyset}$. This is given by
\begin{equation}\label{topclass}
(\zeta_{\emptyset})_w=\casetwo{\prod_{\alpha\in\Phi^+}y_{-\alpha}}{w={\rm id}}{0}{w\ne{\rm id}}
\end{equation}
In fact, the following more general result holds:
\begin{equation}\label{topclassgen}
(\zeta_{I_v})_w=\casetwo{\prod_{\alpha\in \Phi^+\cap w\Phi^+}y_{-\alpha}}{v=w}{0}{w\not\le v}
\end{equation} 

We now calculate the hyperbolic Bott-Samelson classes in the rank 2 cases $A_2$ and $C_2$; note that the root system $B_2$ is the same as $C_2$, so it suffices to consider the latter. 

\begin{example}\label{bsa2} We start with type $A_2$. We use the notation $[ij]:=y_{-\alpha_{ij}}$, for the root $\alpha_{ij}:=\varepsilon_i-\varepsilon_j$. We use the following representation of the (right weak Bruhat order on the) symmetric group $S_3$. 
{\tiny
\[\xymatrix{
  &\rm{id} \ar@{->}[dr] \ar@{->}[dl] \\
 s_1 \ar@{->}[d] & &s_2 \ar@{->}[d]  \\
 s_1s_2 \ar@{->}[dr] & &s_2s_1 \ar@{->}[dl] \\
 &s_1s_2s_1 
 }
\]}
Based on \eqref{topclass}, \eqref{inc2}, \eqref{actiongkm}, and \eqref{eq1lem0} below (with $\alpha=\alpha_{12}$ and $\beta=\alpha_{23}$, so $\alpha+\beta=\alpha_{13}$), we calculate the values of $\zeta_{\emptyset}$, $\zeta_{1}$, $\zeta_{1,2}$, and $\zeta_{1,2,1}$ on the elements of $S_3$.
\begin{align} &\xymatrix{
  & [12][13][23] \ar@{->}[dr] \ar@{->}[dl] \\
 0 \ar@{->}[d] & &  0 \ar@{->}[d]  & \stackrel{Y_1}{\longrightarrow}\qquad\\
 0 \ar@{->}[dr] & & 0 \ar@{->}[dl] \\
 & 0
 }
\xymatrix{
  & [13][23] \ar@{->}[dr] \ar@{->}[dl] \\
 [13][23] \ar@{->}[d]  & &  0 \ar@{->}[d]  & \stackrel{Y_2}{\longrightarrow}\\
0  \ar@{->}[dr] & &  0\ar@{->}[dl] \\
 & 0
 }
\nonumber \\[2mm]
&\xymatrix{
  &  [13] \ar@{->}[dr] \ar@{->}[dl] \\
 [23] \ar@{->}[d]  & & [13]  \ar@{->}[d] & \stackrel{Y_1}{\longrightarrow}\qquad\\
[23]  \ar@{->}[dr] & & 0 \ar@{->}[dl] \\
 & 0
 }
\xymatrix{
  & 1+u[13][23] \ar@{->}[dr] \ar@{->}[dl] \\
 1+u[13][23] \ar@{->}[d]  & & 1  \ar@{->}[d]  \\
 1 \ar@{->}[dr] & & 1 \ar@{->}[dl] \\
 & 1
 }
\label{calcs3}
\end{align}
Similarly, we calculate $\zeta_{2,1,2}$.
\[\xymatrix{
  & 1+u[12][13] \ar@{->}[dr] \ar@{->}[dl] \\
 1 \ar@{->}[d]  & & 1+u[12][13]  \ar@{->}[d]  \\
 1 \ar@{->}[dr] & & 1 \ar@{->}[dl] \\
 & 1
 }\]
\end{example}

\begin{example}\label{bsc2} We consider type $C_2$ with the simple roots $\alpha_0:=2\varepsilon_1$ and $\alpha_1:=\varepsilon_2-\varepsilon_1$. We use the notation $[ij]:=y_{-(\varepsilon_j-\varepsilon_i)}$, for $i\ne j$ in $\{\pm 1,\,\pm 2\}$, where $\overline{\imath}:=-i$ and $\varepsilon_{\overline{\imath}}:=-\varepsilon_i$; in particular, $[\overline{\imath}i]:=y_{-2\varepsilon_i}$ for $i\in\{1,2\}$ and $[\overline{1}2]:=y_{-(\varepsilon_1+\varepsilon_2)}$. We use the following representation of the (right weak Bruhat order on the) hyperoctahedral group $B_2$. 
{\tiny
\[\xymatrix{
  &\rm{id} \ar@{->}[dr] \ar@{->}[dl] \\
 s_0 \ar@{->}[d] & &s_1 \ar@{->}[d]  \\
 s_0s_1 \ar@{->}[d] & &s_1s_0 \ar@{->}[d] \\
s_0s_1s_0 \ar@{->}[dr] & &s_1s_0s_1 \ar@{->}[dl] \\
 &s_0s_1 s_0s_1 
 }
\]}
Based on \eqref{topclass}, \eqref{inc2}, \eqref{actiongkm}, and \eqref{eq1lem0} below, we calculate the values of $\zeta_{\emptyset}$, $\zeta_{0}$, $\zeta_{0,1}$, $\zeta_{0,1,0}$, and $\zeta_{0,1,0,1}$ on the elements of the group $B_2$.
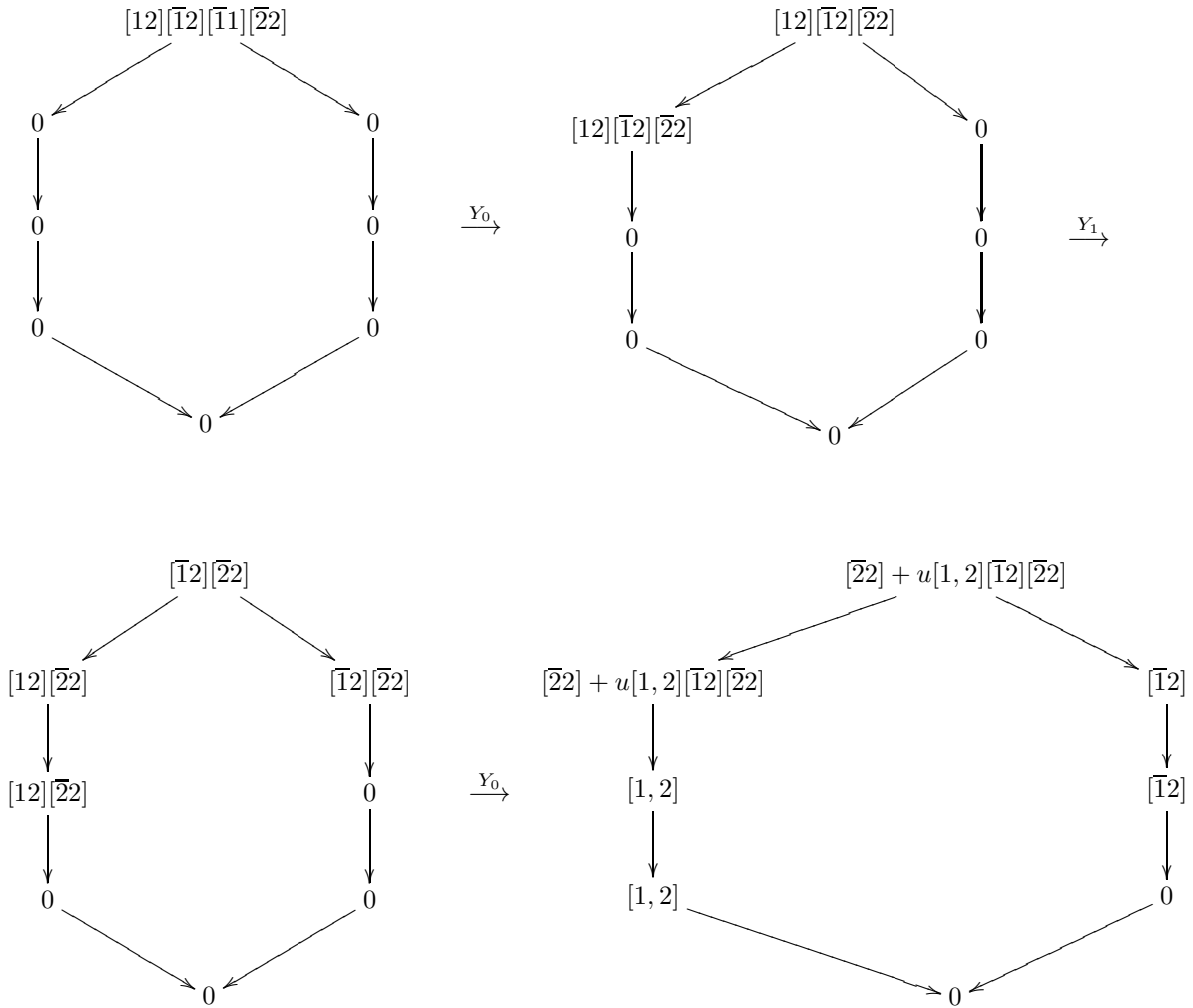
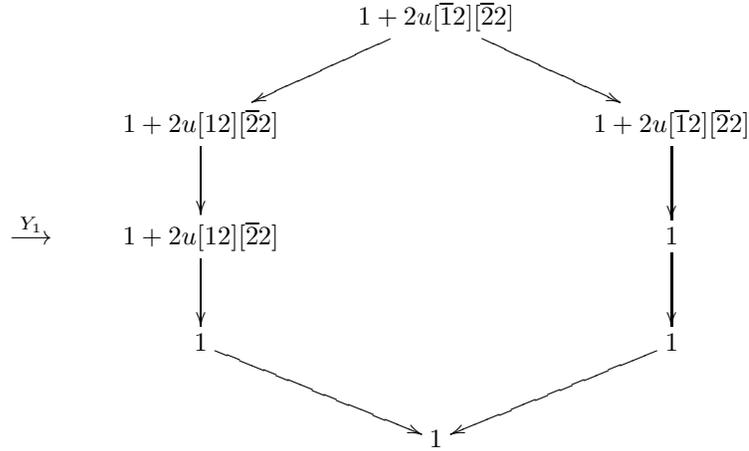
\begin{align} &
\xymatrix{
  &  [12][\overline{1}2][\overline{1}1][\overline{2}2]  \ar@{->}[dr] \ar@{->}[dl] \\
   0   \ar@{->}[d] & &  0   \ar@{->}[d]  \\
   0   \ar@{->}[d] & &  0   \ar@{->}[d] & \stackrel{Y_0}{\longrightarrow}\qquad \\
   0   \ar@{->}[dr] & &  0  \ar@{->}[dl] \\
 & 0
 }
\xymatrix{
  &   [12][\overline{1}2][\overline{2}2]  \ar@{->}[dr] \ar@{->}[dl] \\
   [12][\overline{1}2][\overline{2}2]   \ar@{->}[d] & &  0   \ar@{->}[d]  \\
   0   \ar@{->}[d] & &  0   \ar@{->}[d] & \stackrel{Y_1}{\longrightarrow}\qquad \\
   0   \ar@{->}[dr] & &  0  \ar@{->}[dl] \\
 & 0
 }\label{calcc2}
\end{align}
\begin{align*}
&\xymatrix{
  &   [\overline{1}2][\overline{2}2]  \ar@{->}[dr] \ar@{->}[dl] \\
  [{1}2][\overline{2}2]    \ar@{->}[d] & &  [\overline{1}2][\overline{2}2]   \ar@{->}[d]  \\
     [{1}2][\overline{2}2]  \ar@{->}[d] & &  0   \ar@{->}[d] & \!\!\!\!\!\stackrel{Y_0}{\longrightarrow} \;\;\\
   0   \ar@{->}[dr] & &  0  \ar@{->}[dl] \\
 & 0
 }
\xymatrix{
  &  [\overline{2}2]+u[1,2][\overline{1}2][\overline{2}2]   \ar@{->}[dr] \ar@{->}[dl] \\
   [\overline{2}2]+u[1,2][\overline{1}2][\overline{2}2]   \ar@{->}[d] & &   [\overline{1}2]  \ar@{->}[d]  \\
   [1,2]   \ar@{->}[d] & &  [\overline{1}2]   \ar@{->}[d] \\
   [1,2]   \ar@{->}[dr] & &  0  \ar@{->}[dl] \\
 & 0
 }\end{align*}
\begin{align*} 
&\xymatrix{
&&  &  1+2u[\overline{1}2][\overline{2}2]   \ar@{->}[dr] \ar@{->}[dl] \\
&&   1+2u[{1}2][\overline{2}2]   \ar@{->}[d] & & 1+2u[\overline{1}2][\overline{2}2]    \ar@{->}[d] \\
&\stackrel{Y_1}{\longrightarrow}\!\!  &  1+2u[{1}2][\overline{2}2]   \ar@{->}[d] & & 1   \ar@{->}[d] \\
&  &  1  \ar@{->}[dr] & &  1  \ar@{->}[dl] \\
& & & 1
 }
\end{align*}
More explicitly, to calculate $(\zeta_{0,1,0})_{\rm id}=(\zeta_{0,1,0})_{s_0}$, we use \eqref{eq1lem0} with $\alpha=2\varepsilon_1$ and $\beta=\varepsilon_2-\varepsilon_1$, so $\alpha+\beta=\varepsilon_1+\varepsilon_2$. Then, to calculate $(\zeta_{0,1,0,1})_{\rm id}=(\zeta_{0,1,0,1})_{s_1}$, we use \eqref{eq1lem0} with $\alpha=\varepsilon_2-\varepsilon_1$ and $\beta=\varepsilon_1+\varepsilon_2$, so $\alpha+\beta=2\varepsilon_2$; similarly for $(\zeta_{0,1,0,1})_{s_0}=(\zeta_{0,1,0,1})_{s_0s_1}$, but with $\alpha$ and $\beta$ switched. 

On another hand, $\zeta_1$, $\zeta_{1,0}$, and $\zeta_{1,0,1}$ are computed as follows:
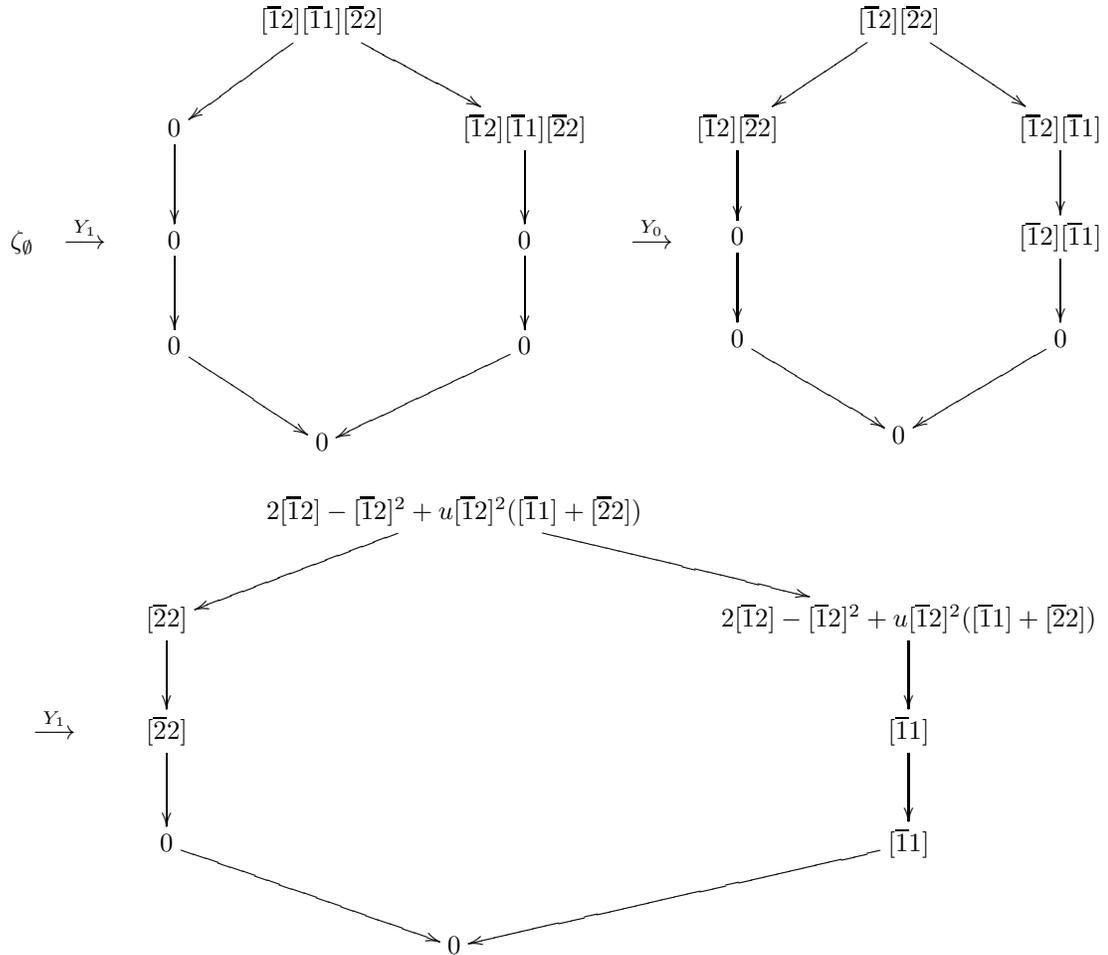
\begin{align} &
\xymatrix{
&&  &   [\overline{1}2][\overline{1}1][\overline{2}2]  \ar@{->}[dr] \ar@{->}[dl] \\
&&   0   \ar@{->}[d] & &  [\overline{1}2][\overline{1}1][\overline{2}2]  \ar@{->}[d]  \\
&\zeta_{\emptyset}\;\;\;\stackrel{Y_1}{\longrightarrow}\!\!\!\!  &   0   \ar@{->}[d] & &  0   \ar@{->}[d] & \!\!\!\!\!\!\!\!\stackrel{Y_0}{\longrightarrow}\; \\
&&   0   \ar@{->}[dr] & &  0  \ar@{->}[dl] \\
&& & 0
 }
\xymatrix{
  &   [\overline{1}2][\overline{2}2]  \ar@{->}[dr] \ar@{->}[dl] \\
  [\overline{1}2][\overline{2}2]    \ar@{->}[d] & &  [\overline{1}2][\overline{1}1]   \ar@{->}[d]  \\
     0  \ar@{->}[d] & &   [\overline{1}2][\overline{1}1]  \ar@{->}[d] & \\
   0   \ar@{->}[dr] & &  0  \ar@{->}[dl] \\
 & 0
 }\label{calcc21}
\end{align}
\begin{align*} 
&\xymatrix{
&&  &  2[\overline{1}2]-[\overline{1}2]^2+u[\overline{1}2]^2([\overline{1}1]+[\overline{2}2])   \ar@{->}[dr] \ar@{->}[dl] \\
&&   [\overline{2}2]   \ar@{->}[d] & & 2[\overline{1}2]-[\overline{1}2]^2+u[\overline{1}2]^2([\overline{1}1]+[\overline{2}2])  \ar@{->}[d] \\
&\stackrel{Y_1}{\longrightarrow}\!\!  &  [\overline{2}2]   \ar@{->}[d] & &  [\overline{1}1]  \ar@{->}[d] \\
&  &  0  \ar@{->}[dr] & &  [\overline{1}1]  \ar@{->}[dl] \\
& & & 0
 }
\end{align*}
More explicitly, to calculate $(\zeta_{1,0,1})_{\rm id}=(\zeta_{1,0,1})_{s_1}$, we use \eqref{eq2lem0} with $\alpha=\varepsilon_2-\varepsilon_1$ and $\beta=2\varepsilon_1$, so $2\alpha+\beta=2\varepsilon_2$. 

A more involved computation, based on the expression for $\zeta_{1,0,1}$ in \eqref{calcc21}, leads to the following expression for $\zeta_{1,0,1,0}$:
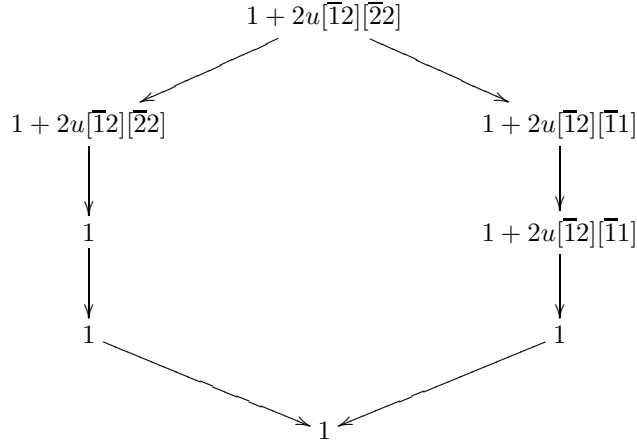
\[\xymatrix{
  &   1+2u[\overline{1}2][\overline{2}2]  \ar@{->}[dr] \ar@{->}[dl] \\
  1+2u[\overline{1}2][\overline{2}2]    \ar@{->}[d] & &  1+2u[\overline{1}2][\overline{1}1]   \ar@{->}[d]  \\
     1  \ar@{->}[d] & &   1+2u[\overline{1}2][\overline{1}1]  \ar@{->}[d] & \\
   1   \ar@{->}[dr] & &  1 \ar@{->}[dl] \\
 & 1
 }
\]
Note that we need to apply four times the formula for $y_{\alpha+\beta}$ in \eqref{fglcalc} only to calculate $(\zeta_{1,0,1,0})_{\rm id}=(\zeta_{1,0,1,0})_{s_0}$.
\end{example}

\begin{example}\label{bsa3} We now consider type $A_3$ and, for simplicity, we set $\mu_1=0$ in \eqref{defell}, i.e., we consider the {\em Lorentz formal group law}. We show $(\zeta_{I_{w_\circ}})_{\rm{id}}$ for some reduced words for $w_\circ=4321$:
\begin{align*}
&(\zeta_{1,2,3,1,2,1})_{\rm{id}}=(\zeta_{1,2,1,3,2,1})_{\rm{id}}=1+2u[14][24]+u^2[13][14][23][24]\,,\\
&(\zeta_{1,2,3,2,1,2})_{\rm{id}}=(\zeta_{2,1,2,3,2,1})_{\rm{id}}=1+u[13][14]+u[14][24]+u^2[13][14][24][34]\,.
\end{align*}
\end{example}

\section{Main result}

An important open problem in Schubert calculus beyond $K$-theory is defining Schubert classes which are independent of a reduced word for the indexing Weyl group element. The standard topological approach works if the Schubert variety $X(w)$ is smooth, and the corresponding class $[X(w)]$ has a simple formula; namely, by \cite[Theorem~7.2.1]{BL}, in the GKM model of $E_T^*(G/B)$ (discussed in Section~\ref{gkmsetup}), we have:
\begin{equation}\label{smooth}
[X(w)]_v=\frac{\displaystyle{\prod_{\beta\in\Phi^+}y_{-\beta}}}{\displaystyle{\prod_{\stackrel{\beta\in \Phi^+}{s_\beta v\le w}} y_{-\beta} }}\,,  
\end{equation}
for $v\le w$ in the Weyl group $W$, and otherwise $[X(w)]_v=0$, cf. \eqref{topclassgen}. 

We propose an approach in elliptic cohomology based on the {\em Kazhdan-Lusztig basis} of the corresponding {\em Hecke algebra}, and provide support for it. 

Let us explain the motivation. In cohomology and $K$-theory, by applying the operator $Y_{w_\circ}$ to the class of a point gives the fundamental class of the flag manifold (i.e., the identity element). However, this does not happen beyond $K$-theory, see Examples~\ref{bsa2} and \ref{bsa3}. Our goal is to modify $Y_{I_w}$ in such a way that: (1) the new operator corresponding to $w$ does not depend on a reduced word for $w$; (2) when applied to the class of a point, the operator corresponding to $w_\circ$ gives $1$. 

We start our construction by recalling the formal Demazure algebra ${\mathcal D}_F$ from Section~\ref{fda}. By \linebreak \cite[Thm.~5.4]{HMSZ}, in type $A$, if $F$ is the hyperbolic formal group law, the algebra  ${\mathcal D}_F$ is generated by the elements $Y_i$ and multiplications by $z\in S$ subject to the relations \eqref{square}, \eqref{braid2}, \eqref{braid3}, and the following one:
\begin{equation}
Y_iz=s_i(z)Y_i+\mu_1 z-Y_i(z)\,.
\end{equation}
Following \cite[Def.~5.3]{HMSZ}, let $\mathbf{D}_F$ denote the $R$-subalgebra of ${\mathcal D}_F$ generated by the elements $Y_i$ only. In \cite[Prop.~6.1]{HMSZ} it was shown that for the additive formal group law $F=F_a$ (respectively, the multiplicative one $F=F_m$), the algebra $\mathbf{D}_F$ is isomorphic to the nil-Hecke algebra (respectively the 0-Hecke algebra) of Kostant-Kumar \cite{KK86,KK90}.

We will now recall the Hecke algebra $\mathcal H$ of the Weyl group $W$, and we refer to \cite{hum} for more information. Instead of the usual generators $T_i$, we use those in \cite[Section~1]{Ca}, namely $\tau_i:=tT_i$, where $t=q^{-1/2}$. So $\mathcal H$ is the $\Z[t^{\pm 1}]$-algebra generated by $\tau_i$ subject to 
\begin{equation}\label{hecke}
(\tau_i+t)(\tau_i-t^{-1})=0\;\;\Longleftrightarrow\;\;\tau_i^2=(t^{-1}-t)\tau_i+1\,,
\end{equation}
 and the usual braid relations.

From now on we work in the setup of the hyperbolic formal group law \eqref{defell} with $\mu_1=1$ and $\mu_2=-(t+t^{-1})^{-2}=-u$, so we let the base ring $R$ be $\Z[t^{\pm 1},\,(t+t^{-1})^{-1}]$. This case does not correspond to a complex oriented theory \cite[\S4]{BK91}; however, since we work in the algebraic setup, we will still use the notation $\HY_T^*(G/B)$.

In \cite[Prop.~8.2]{CZZ1} (in type $A$) and in \cite{LNZ} (in arbitrary type), it is shown that ${\mathcal H}\otimes_{\Z[t^{\pm 1}]} R$ is isomorphic to the corresponding algebra $\mathbf{D}_F$ via 
\begin{equation}\label{taui}\tau_i\mapsto (t+t^{-1})Y_i-t\,.\end{equation}
We identify the two algebras, and note that the involution on $\mathcal{H}$ (sending $t\mapsto t^{-1}$
and $\tau_i\mapsto \tau_i^{-1}$) corresponds to the
involution on $\mathbf{D}_F$ obtained by extending the involution
$t\mapsto t^{-1}$ on the coefficient ring. Indeed, each
push-pull element $Y_i=\tfrac{1}{t+t^{-1}}(\tau_i+t)$ is invariant under this involution; see below.

Consider the Kazhdan-Lusztig basis $\{\gamma_w\::\:w\in W\}$ of ${\mathcal H}$, given by 
\begin{equation}\label{klb}\gamma_w=\tau_w+\sum_{v<w}t\,\pi_{v,w}(t)\,\tau_v\,,\end{equation}
where $\pi_{v,w}(t)$ are the {\em Kazhdan-Lusztig polynomials} (in terms of the classical notation, we have $P_{v,w}(t)=t^{-(\ell(w)-\ell(v)-1)}\pi_{v,w}(t)$). Recall that one of its defining properties is its invariance under the above involution. We implicitly use the well-known result of Kazhdan-Lusztig (see, e.g., \cite[Section~6.1]{BL}) that the Schubert variety $X(w)$ is {\em rationally smooth} (which is implied by being smooth) if and only if $P_{v,w}(t)=1$ for all $v\le w$. 

We will also use the iterative construction of the Kazhdan-Lusztig basis, see, e.g., \cite{hum}[Section~7.11]:
\begin{equation}\label{iterkl}\gamma_{w}=\gamma_{s_i}\gamma_{v}-\sum_{\stackrel{z\prec v}{s_iz<z}}\mu(z,v)\,\gamma_z\,,\end{equation}
where $v=s_iw<w$. Here $z\prec v$ means that $z<v$ in Bruhat order and the largest allowable degree $(\ell(v)-\ell(z)-1)/2$ of the Kazhdan-Lusztig polynomial $P_{z,v}(q)$ is attained; furthermore, $\mu(z,v)$ is the coefficient of this leading term. In particular, $\ell(w)-\ell(z)$ is even.

For $w\in W$, denote by $\Gamma_w$ the element of $\mathbf{D}_F$ which corresponds to $\gamma_w$ via the above isomorphism between ${\mathcal H}\otimes_{\Z[t^{\pm 1}]} R$ and $\mathbf{D}_F$. In particular, $\gamma_{s_i}=\tau_i+t$, which explains why $Y_i=\tfrac{1}{t+t^{-1}}\Gamma_{s_i}$ is invariant under the involution. 

\begin{example}\label{norep} Let $w$ be a product of distinct simple reflections $w=s_{i_1}\ldots s_{i_n}$. It is well-known that
\[\gamma_w=\gamma_{s_{i_1}}\ldots\gamma_{s_{i_n}}\,,\;\;\;\;\mbox{so}\;\:(t+t^{-1})^{-\ell(w)}\Gamma_w= Y_w\,;\]
here the notation $Y_w$ indicates that the corresponding product of operators $Y_i$ is also independent of a reduced word for $w$. 
\end{example}

The next examples show the very close relationship between the twisted braid relation in the hyperbolic case and the Kazhdan-Lusztig basis, in the sense that the former expresses the invariance of $(t+t^{-1})^{-\ell(w_\circ)} \Gamma_{w_\circ}$ for $w_\circ$ in the rank 2 Weyl groups (with respect to the two reduced words for $w_\circ$).

\begin{example}\label{a1a2} Consider type $A_2$. 
As we have seen in the previous example:
\[ \Gamma_{s_i}=(t+t^{-1})Y_i\,,\;\;\;\; \Gamma_{s_is_j}=(t+t^{-1})^2Y_iY_j\,.\]
It is also well-known that 
\begin{equation}\label{a2xx}\gamma_{w_\circ}=\gamma_{s_is_js_i}=\gamma_{s_i}\gamma_{s_j}\gamma_{s_i}-\gamma_{s_i}=\gamma_{s_j}\gamma_{s_i}\gamma_{s_j}-\gamma_{s_j}\,.\end{equation}
Therefore, we have
\[\Gamma_{s_is_js_i}=(t+t^{-1})^3(Y_iY_jY_i-u Y_i)=(t+t^{-1})^3(Y_jY_iY_j-u Y_j)\,.\]
Thus, the independence of $(t+t^{-1})^{-3}\Gamma_{w_\circ}$ from a reduced word for $w_\circ=s_1s_2s_1=s_2s_1s_2$ is given by the twisted braid relation \eqref{braid3}.
\end{example}

\begin{example}\label{b2g2} Let us now turn to types $B_2$ and $G_2$, and calculate the corresponding element $\gamma_{w_\circ}$ based on the recursive formula~\eqref{iterkl}. We illustrate this computation in type $G_2$, which is more involved. We observe, in each case, that the equality of the two obtained expressions for $(t+t^{-1})^{-\ell(w_\circ)} \Gamma_{w_\circ}$ is precisely the corresponding twisted braid relation, namely \eqref{braid4} and \eqref{braid6}. 

In type $B_2$ we have
\begin{equation}\label{b2xx}\gamma_{w_\circ}=\gamma_{s_is_js_is_j}=\gamma_{s_i}\gamma_{s_j}\gamma_{s_i}\gamma_{s_j}-2\gamma_{s_i}\gamma_{s_j}\,,\end{equation}
which implies
\[(t+t^{-1})^{-4}\Gamma_{s_is_js_is_j}=Y_iY_jY_iY_j-2u Y_iY_j=Y_jY_iY_jY_i-2u Y_jY_i\,.\]

In type $G_2$ we have
\[\gamma_{w_\circ}=\gamma_{s_is_js_is_js_is_j}=\gamma_{s_i}\gamma_{s_j}\gamma_{s_i}\gamma_{s_j}\gamma_{s_i}\gamma_{s_j}-4\gamma_{s_i}\gamma_{s_j}\gamma_{s_i}\gamma_{s_j}+3\gamma_{s_i}\gamma_{s_j}\,.\]
Indeed, we combine the analogues of \eqref{a2xx} and \eqref{b2xx}, as well as the recursive formulas
\begin{align*}\gamma_{w_\circ}&=\gamma_{s_i}\gamma_{s_js_is_js_is_j}-\mu(s_is_js_is_j,s_js_is_js_is_j)\gamma_{s_is_js_is_j}-\mu(s_is_j,s_js_is_js_is_j)\gamma_{s_is_j}\,,\\
\gamma_{s_js_is_js_is_j}&=\gamma_{s_j}\gamma_{s_is_js_is_j}-\mu(s_js_is_j,s_is_js_is_j)\gamma_{s_js_is_j}-\mu(s_j,s_is_js_is_j)\gamma_{s_j}\,;
\end{align*}
in both relations, the value of the first $\mu$-coefficient is $1$ and of the second one is $0$, due to the well-known fact that the (non-zero) Kazhdan-Lusztig polynomials for dihedral groups are all equal to $1$. 
Like in Example~\ref{a1a2}, we conclude that
\[(t+t^{-1})^{-6}\Gamma_{s_is_js_is_js_is_j}=Y_iY_jY_iY_jY_iY_j-4u Y_iY_jY_iY_j+3u^2Y_iY_j=Y_jY_iY_jY_iY_jY_i-4u Y_jY_iY_jY_i+3u^2Y_jY_i\,.\]
\end{example}

 We generalize the above examples, by giving some information about the expansion of $\Gamma_w$ in a $Y$-basis of $\mathbf{D}_F$. A priori, by combining \eqref{klb} and \eqref{taui}, this expansion contains powers of $t$ in addition to powers of $t+t^{-1}$, but it turns out that a more precise description can be given. 

\begin{prop}\label{combin} Let $I_w$ be a reduced word for $w$. The element $(t+t^{-1})^{-\ell(w)}\,\Gamma_{w}$ of $\mathbf{D}_F$ has an expansion of the form
\[Y_{I_w}+\sum_{\stackrel{v<w}{\ell(w)-\ell(v)\in 2\Z}} c_v\, u^{(\ell(w)-\ell(v))/2}\, Y_{I_v}\,,\]
for some integer coefficients $c_v$, and some reduced subwords $I_v$ of $I_w$. 
\end{prop}

\begin{proof} Iterating \eqref{iterkl}, we express any $\gamma_w$ in terms of sums of products of $\gamma_{s_i}$. Note that we can arrange things such that all these products correspond to reduced subwords of $I_w$, and also that the parity of the number of factors in them is the same. The result follows by substituting $\Gamma_{s_i}=(t+t^{-1})Y_i$, and by dividing through by $(t+t^{-1})^{\ell(w)}$. 
\end{proof}

Proposition \ref{combin} suggests the following definition for our Schubert classes, which we call {\em Kazhdan-Lusztig (KL-) Schubert classes}.

\begin{defi}
Consider the element ${\mathfrak S}_w$ in $\HY_T^*(G/B)$ given by $(t+t^{-1})^{-\ell(w)}\,\Gamma_{w^{-1}}(\zeta_{\emptyset})$ under the action of $\mathbf{D}_F$ on the GKM model of $\HY_T^*(G/B)$.
\end{defi}

Consider the limit $t\to 0$, which implies $u\to 0$; so the formal group law becomes the multiplicative one, for $K$-theory. We obtain the following corollary of Proposition \ref{combin}.

\begin{cor}\label{schubcl} In the limit $t\to 0$, the KL-Schubert class ${\mathfrak S}_w$ becomes the Bott-Samelson class $\zeta_w$ in $K$-theory (which is known to be independent of a reduced word for $w$, and to coincide with the corresponding Schubert class, defined topologically).
\end{cor}

The following corollary provides additional motivation for the KL-Schubert classes.

\begin{cor}
The classes $\{{\mathfrak S}_w\::\:w\in W\}$ form a basis of $\HY_T^*(G/B)$.
\end{cor}

\begin{proof} Fix a reduced word $I_v$ for each $v$ in $W$. By using the twisted braid relation, we can convert the expansion of $\Gamma_w$ in Proposition~\ref{combin} into one in the basis $\{Y_{I_v}\}$ of $\mathbf{D}_F$. It follows that the transition matrix from the KL-Schubert classes ${\mathfrak S}_w$ to the basis of Bott-Samelson classes $\{\zeta_{I_v}\::\:v\in W\}$ is triangular with $1$'s on the diagonal. 
\end{proof}

Let us now calculate the KL-Schubert classes in ranks 1 and 2, with the exception of type $G_2$. 

\begin{example}\label{sa1a2} For type $A_1$ it is immediate that ${\mathfrak S}_{s_1}=Y_1\,\zeta_{\emptyset}=1$. For the rank 2 cases, we use the formulas for $(t+t^{-1})^{-\ell(w)}\Gamma_w$ in Examples~\ref{a1a2} and \ref{b2g2}. In types $A_2$ and $C_2$, the elements ${\mathfrak S}_{s_i}$ and ${\mathfrak S}_{s_is_j}$ coincide with the corresponding Bott-Samelson classes, given by \eqref{calcs3}, \eqref{calcc2}, and \eqref{calcc21}. 
Furthermore, in both cases we also have ${\mathfrak S}_{w_\circ}=1$, based on \eqref{calcs3} and \eqref{calcc2}. This fact is stated for all types $A_{n-1}$ and $C_n$ in Theorem~\ref{mainthm}~(2).

Also note that, in type $C_2$, the KL-Schubert class ${\mathfrak S}_{s_0s_1s_0}$ is obtained by setting $u=0$ in the Bott-Samelson class $\zeta_{0,1,0}$, expressed in \eqref{calcc2}. The simple expressions for the above KL-Schubert classes are not surprising because in all these cases the corresponding Schubert varieties are non-singular; see Remark~\ref{nsrank2} and \cite{BL}. The only singular Schubert variety in type $C_2$ is $X(s_1s_0s_1)$, with maximal singular locus $X(s_1)$; correspondingly, we calculate based on \eqref{calcc21}:
\begin{equation}\label{exkls}
({\mathfrak S}_{s_1s_0s_1})_{\rm id}=({\mathfrak S}_{s_1s_0s_1})_{s_1}=2[\overline{1}2]-[\overline{1}2]^2+u[\overline{1}2]([\overline{1}2][\overline{1}1]+[\overline{1}2][\overline{2}2]-[\overline{1}1][\overline{2}2])\,.
\end{equation}
The other values of ${\mathfrak S}_{s_1s_0s_1}$ are the same as those of the Bott-Samelson class $\zeta_{1,0,1}$, expressed in \eqref{calcc21}. 
\end{example}

In \cite{LZ}[Conjecture 6.4] we conjectured a positivity property for the (hyperbolic) Bott-Samelson classes. Here we conjecture the same property for the KL-Schubert classes. 

\begin{conj}\label{posconj}
The evaluation $({\mathfrak S}_{v})_{w}$, for any $w\le v$, can be expressed as a (possibly infinite) sum of monomials in $y_{-\alpha}$, where $\alpha$ are positive roots, such that the coefficient of each monomial is of the form 
$$(-1)^{k-(N-\ell(v))}\, c\, u^{(m-k)/2}\,,$$
 where $c$ is a positive integer, $m$ is the degree of the monomial, $N-\ell(v)\le k\le m$, $m-k$ is even, and $N$ is the number of positive roots. 
\end{conj}

\begin{rem}\label{easypos} The above positivity property is a generalization of the one in $K$-theory which is made explicit in Graham's formula \cite{graekt} for the localization of Schubert classes at torus fixed points, cf. also~\cite{LZ}.
\end{rem}
 
\begin{example}\label{posex}
Let us check the positivity property for $({\mathfrak S}_{s_1s_0s_1})_{\rm id}=({\mathfrak S}_{s_1s_0s_1})_{s_1}$ calculated in \eqref{exkls}. Based on the formula for $y_{\alpha+\beta}$ in \eqref{fglcalc} with $\alpha=[12]$ and $\beta=[\overline{1}1]$, so that $\alpha+\beta=[\overline{1}2]$, we re-express the mentioned evaluation in a ``positive'' form. In fact, it suffices to focus on the following subexpression:
\begin{align*}&u[\overline{1}2]([\overline{1}2][\overline{1}1]+[\overline{1}2][\overline{2}2]-[\overline{1}1][\overline{2}2])=\\
=&u[\overline{1}2]([\overline{1}2][\overline{1}1]+([12]+[\overline{1}1]-[12][\overline{1}1]+u[12][\overline{1}1][\overline{1}2])[\overline{2}2]-[\overline{1}1][\overline{2}2])\\
=&u[\overline{1}2]^2[\overline{1}1]+u[12][\overline{1}2][\overline{2}2]-u[12][\overline{1}2][\overline{1}1][\overline{2}2]+u^2[12][\overline{1}2]^2[\overline{1}1][\overline{2}2]\,.
\end{align*}
\end{example}

Furthermore, a natural property that the Schubert classes should have, which we conjecture for our classes ${\mathfrak S}_w$, is the following. 

\begin{conj}\label{mainconj}
If the Schubert variety $X(w)$ is smooth, then the class $[X(w)]$ given by {\rm \eqref{smooth}} coincides with ${\mathfrak S}_w$. 
\end{conj}

\begin{rem}\label{nsrank2} Examples~\ref{posex} and \ref{sa1a2} show that Conjectures~\ref{posconj} and \ref{mainconj} are true in ranks 1 and 2, with the exception of type $G_2$.
\end{rem}

We tested Conjecture~\ref{mainconj} on the computer. Moreover, we proved the special cases below; here $W_n$ denotes the Weyl group of a root system of type $A_{n-1}$ or $C_n$, and we use the standard embedding of $W_{n-1}$ into $W_n$. 

\begin{theo}\label{mainthm}
Conjecture {\rm \ref{mainconj}} is true in the following cases (which all correspond to non-singular Schubert varieties): 
\begin{enumerate}
\item[{\rm (1)}] in all types for $w$ which is a product of distinct simple reflections;
\item[{\rm (2)}] in the classical types $A_{n-1}$ and $C_n$, for $w^{-1}$, where $w$ is a highest coset representative for $W_n/W_{n-1}$ (in particular, for $w=w_\circ$, the longest element, we have ${\mathfrak S}_{w_\circ}=1$). 
\end{enumerate}
\end{theo}

\begin{rems} (1) Cases (1), which are concerned with $w$ of ``small'' length, do not involve the formal group law in the recursive calculation of the related Bott-Samelson classes $\zeta_{I_w}$ (see the proof below); this fails as soon as the reduced words for $w$ have repeated simple reflections. 

(2) Cases (2), which are concerned with $w$ of ``large'' length, are highly non-trivial because, for instance, the various Bott-Samelson classes $\zeta_{I_{w}}$ have more and more involved expressions as $\ell(w)$ increases, see Example~\ref{bsa3}. However, the Kazhdan-Lusztig operator $\Gamma_{w}$ combines several classes $\zeta_{I_{v}}$ for $v\le w$ such that the final result is simple, cf. Proposition~\ref{combin}.
\end{rems}

\section{Proofs}

We now prove Theorem \ref{mainthm} in several steps. The following lemma will be useful.

\begin{lem}\label{prodcl} {\rm (1)} If $\prod_\alpha y_{-\alpha}= \prod_\beta y_{-\beta}$ in the cohomology or $K$-theory algebra $S$, for two subsets of the positive roots, then the two subsets are the same. 

{\rm (2)} Assume that the Schubert variety $X(w)$ is non-singular. If the evaluation of ${\mathfrak S}_w$ at every element of $W$ is a product $\prod_\alpha y_{-\alpha}$ over some subset of the positive roots, then ${\mathfrak S}_w=[X(w)]$.
\end{lem}

\begin{proof} It is well-known that the cohomology algebra $S$ can be identified with the completion of the symmetric algebra
$Sym_\Z(\Lambda)$ via $y_{-\lambda}=\lambda$; in fact, $Sym_\qq(\Lambda)\simeq\qq[x_1,\ldots,x_n]$, where $x_i=\alpha_i$. The claim now follows from the
fact that the latter is a unique factorization domain and no positive multiple of a root is another root (except for the trivial case). In the $K$-theory case, we reduce to the cohomology case by using the fact that the cohomology algebra $S$ is the associated graded algebra to the $K$-theory one. The second part follows immediately from Corollary~\ref{schubcl} and the first part.
\end{proof}

\begin{proof}[Proof of Theorem {\rm \ref{mainthm} (1)}] The corresponding Schubert varieties $X(w)$ are non-singular by a criterion of Fan \cite{fan}. By Example~\ref{norep}, a class ${\mathfrak S}_w$ is computed in this case by $Y_w$, so it coincides with $\zeta_w$. It is easy to see that this recursive computation only involves the division of a product of elements $y_{-\alpha}$ by one of the factors (in particular, the formal group law is not involved in the computation, see the application of the first two operators $Y_i$ in \eqref{calcs3}). Thus, all values $(\zeta_w)_v$ are products of elements $y_{-\alpha}$. The proof is concluded by applying Lemma~\ref{prodcl}.
\end{proof}

We now turn to Theorem \ref{mainthm} (2) for $w=w_\circ$ in type $A_{n-1}$. In other words, we show that
\begin{equation}\label{casew0a} (t+t^{-1})^{-N}\,\Gamma_{w_\circ}(\zeta_{\emptyset})=1\,,\end{equation}
where $N=\ell(w_\circ)=|\Phi^+|$, that is, the Schubert class ${\mathfrak S}_{w_\circ}$ is the fundamental class of the flag variety.
The proof is based on several lemmas.

We use freely the notation introduced in the previous sections, as well as the related background, in particular the GKM model and the formulas \eqref{actiongkm} for the action of the dual Demazure algebra in this model. Consider the root system of type $A_{n-1}$, with roots $\alpha_{ij}:=\varepsilon_i-\varepsilon_j$, and the symmetric group on $\{1,\ldots,n\}$ as Weyl group, denoted $W_n$. We use the one-line notation for permutations, i.e., $w=i_1\ldots i_n$ means that $w(k)=i_k$. Let $[ij]:=y_{-\alpha_{ij}}$ in the formal group algebra $S=S_n$. For any positive integer $k$, define the functions $\rho_k^n\,:\,W_n\rightarrow S_n$:
\begin{equation}\label{defrho}
\rho_k^n(w):=\casetwoct{[i_1n]\ldots[i_{k-1}n]}{w^{-1}(n)\ge k}{0}
\end{equation}
where $w=i_1\ldots i_{k-1}\ldots$. Clearly $\rho_k^n$ is identically $0$ if $k>n$, and \eqref{defrho} is understood to define $\rho_1^n$ as identically $1$. For simplicity, we write $\rho_k$ instead of $\rho_k^n$ for any $k$ in Lemmas \ref{lem2} and \ref{lem1}, as $n$ is fixed.

\begin{lem}\label{lem0} In the formal group algebra $S$, we have
\begin{align}\frac{y_{\alpha+\beta}}{y_\alpha}+\frac{y_\beta}{y_{-\alpha}}&=1+uy_\beta y_{\alpha+\beta} \,,\label{eq1lem0}\\
\frac{y_{2\alpha+\beta}}{y_\alpha}+\frac{y_\beta}{y_{-\alpha}}&=2-y_{\alpha+\beta}+uy_{\alpha+\beta} (y_{\beta}+y_{2\alpha+\beta})\,.
\label{eq2lem0}\end{align}
\end{lem}

\begin{proof} From the definition \eqref{defell} of the hyperbolic formal group law, it follows that
\begin{equation}\label{fglcalc}y_{\alpha+\beta}=y_\alpha+y_\beta-y_\alpha y_\beta+uy_\alpha y_\beta y_{\alpha+\beta}\,,\;\;\;\;\;y_{-\alpha}=\iota(y_\alpha)=\frac{y_\alpha}{y_\alpha-1}\,.\end{equation}
By using these facts, we easily derive
\begin{align*}
\frac{y_{\alpha+\beta}}{y_\alpha}+\frac{y_\beta}{y_{-\alpha}}&=\frac{y_{\alpha+\beta}+(y_\alpha y_\beta-y_\beta)}{y_\alpha}\\
&=\frac{y_\alpha+uy_\alpha y_\beta y_{\alpha+\beta}}{y_\alpha}\\
&=1+uy_\beta y_{\alpha+\beta}\,.
\end{align*}
Similarly, by using the first relation in \eqref{fglcalc} twice, we calculate
\begin{align*}
\frac{y_{2\alpha+\beta}}{y_\alpha}+\frac{y_\beta}{y_{-\alpha}}&=\frac{y_{2\alpha+\beta}+(y_\alpha y_\beta-y_\beta)}{y_\alpha}\\
&=\frac{(y_\alpha+y_{\alpha+\beta}-y_\alpha y_{\alpha+\beta}+uy_\alpha y_{\alpha+\beta}  y_{2\alpha+\beta})+(y_\alpha y_\beta-y_\beta)}{y_\alpha}\\
&=\frac{2y_\alpha+uy_\alpha y_\beta y_{\alpha+\beta}-y_\alpha y_{\alpha+\beta}+uy_\alpha y_{\alpha+\beta} y_{2\alpha+\beta}}{y_\alpha}\\
&=2-y_{\alpha+\beta}+uy_{\alpha+\beta} (y_{\beta}+y_{2\alpha+\beta})\,.
\end{align*}
\end{proof}

\begin{lem}\label{lem2} We have
\begin{equation}\label{lem2f}Y_k\,\rho_{k+1}-u\rho_{k+2}=\rho_k\,,\;\;\;\;\;Y_k\,\rho_{k+2}=\rho_{k+2}\,.\end{equation}
\end{lem}

\begin{proof}
From the definition of $\rho_{k+2}$, we can see that $\rho_{k+2}(w)=\rho_{k+2}(ws_k)$, which means that $\rho_{k+2}$ is $s_k$-invariant. Thus, the second formula in \eqref{lem2f} follows by Remark \ref{symmi}. 

To derive the first formula, we start by considering $w$ with $w^{-1}(n)=k+1$, so $w=i_1\ldots i_k n \ldots$. Since $\rho_{k+1}(ws_k)=0$, we have
\[(Y_k\,\rho_{k+1})(w)=\frac{[i_1 n]\ldots [i_k n]}{[i_k n]}=[i_1 n]\ldots [i_{k-1} n]=(Y_k\,\rho_{k+1})(w s_k)\,.\]
So we have
\begin{equation}\label{pf1}
(Y_k\,\rho_{k+1})(w)=\rho_k(w)=\rho_k(w)+u\rho_{k+2}(w)\,,\;\;\;\;\;\mbox{for $w^{-1}(n)\le k+1$}\,.
\end{equation}
On the other hand, if $w^{-1}(n)\ge k+2$, i.e., $w=i_1\ldots i_k i_{k+1}\ldots$ with all shown entries different from $n$, we calculate as follows, based on \eqref{eq1lem0}:
\begin{align}\label{pf2}
(Y_k\,\rho_{k+1})(w)&=\frac{[i_1 n]\ldots [i_k n]}{[i_k i_{k+1}]}+\frac{[i_1 n]\ldots [i_{k-1} n][i_{k+1} n]}{[i_{k+1} i_{k}]}\\
&=[i_1 n]\ldots [i_{k-1} n]\left(\frac{[i_k n]}{[i_k i_{k+1}]}+\frac{[i_{k+1} n]}{[i_{k+1} i_{k}]}\right)\nonumber\\
&=[i_1 n]\ldots [i_{k-1} n](1+u[i_{k} n][i_{k+1} n])\nonumber\\
&=\rho_k(w)+u\rho_{k+2}(w)\,.\nonumber
\end{align}
The first formula in \eqref{lem2f} follows by combining \eqref{pf1} and \eqref{pf2}. 
\end{proof}

\begin{lem}\label{lem1} We have
\begin{equation*}\tau_k\ldots\tau_{n-1}\,\rho_n=(t+t^{-1})^{n-k}\rho_k-t(t+t^{-1})^{n-k-1}\rho_{k+1}\,.\end{equation*}
\end{lem}

\begin{proof} We proceed by decreasing induction on $k$, with base case $k=n$, and we use freely the formula for the action of $\tau_i$ coming from \eqref{taui}. Assuming the statement for $k+1$, we need to calculate
\begin{equation}\label{pf3}
\tau_k(\tau_{k+1}\ldots\tau_{n-1}\,\rho_n)=(t+t^{-1})^{n-k-1}\tau_k(\rho_{k+1})-t(t+t^{-1})^{n-k-2}\tau_k(\rho_{k+2})\,.
\end{equation}
By Lemma \ref{lem2}, we first have
\[
t\tau_k(\rho_{k+2})=t\left((t+t^{-1})Y_k\,\rho_{k+2}-t\rho_{k+2}\right)=\rho_{k+2}\,.
\]
Based on this, and by using Lemma \ref{lem2} again, we can rewrite the right-hand side of \eqref{pf3} as follows:
\begin{align*}
&\left((t+t^{-1})^{n-k}\, Y_k\,\rho_{k+1}-t(t+t^{-1})^{n-k-1}\rho_{k+1}\right)-(t+t^{-1})^{n-k-2}\rho_{k+2}\\
=&(t+t^{-1})^{n-k}\left(Y_k\,\rho_{k+1}-u\rho_{k+2}\right)-t(t+t^{-1})^{n-k-1}\rho_{k+1}\\
=&(t+t^{-1})^{n-k}\rho_k-t(t+t^{-1})^{n-k-1}\rho_{k+1}\,.
\end{align*}
This concludes the induction proof.
\end{proof}

\begin{proof}[Proof of Theorem {\rm \ref{mainthm} (2)} for $w=w_\circ$ in type $A_{n-1}$]
We proceed by induction on $n$. The statement for $n=2$ and $n=3$ is Example~\ref{sa1a2}. Assuming that it holds for $n-1$, for $n\ge 4$, we will prove it for $n$. We identify the Weyl group $W_{n-1}$ with its image (under the standard embedding) in $W_n$, as a parabolic subgroup, and let $W^{n-1}$ be the set of lowest coset representatives in $W_n/W_{n-1}$. Let $w_\circ^n$ and $w_\circ^{n-1}$ be the longest elements of $W_n$ and $W_{n-1}$, respectively, and $N:=\ell(w_\circ^n)=\binom{n}{2}$. Let $\zeta_\emptyset^n$ and $\zeta_\emptyset^{n-1}$ be the respective classes for the flag varieties corresponding to $SL_n$ and $SL_{n-1}$. Using also the standard embedding of the formal group algebra $S_{n-1}$ into $S_n$, we extend the function $\rho_k^{n-1}\,:\,W_{n-1}\rightarrow S_{n-1}$ defined in \eqref{defrho} to a function from $W_n$ to $S_n$, by defining it to be identically $0$ outside $W_{n-1}$. 

As the Kazhdan-Lusztig polynomials $P_{v,w_\circ^n}(t)$ are $1$ for all $v$ in $W_n$, we have
\begin{equation}\label{klexp}\gamma_{w_\circ^n}=\sum_{w\in W_n} t^{N-\ell(w)}\tau_w\,.\end{equation}
Since every element of $W_n$ can be factored uniquely as $w=uv$ with $u\in W^{n-1}$, $v\in W_{n-1}$, and $\ell(w)=\ell(u)+\ell(v)$, we can rewrite the above expression as
\begin{equation}\label{gamma0}\gamma_{w_\circ^n}=\sum_{u\in W^{n-1}} t^{n-1-\ell(u)}\tau_u\sum_{v\in W_{n-1}} t^{(N-n+1)-\ell(v)}\tau_v=\sum_{u\in W^{n-1}}t^{n-1-\ell(u)}\tau_u\gamma_{w_\circ^{n-1}}\,.\end{equation}
Thus, we need to calculate
\begin{equation}\label{eqxx}
(t+t^{-1})^{-N}\Gamma_{w_\circ^n}(\zeta_\emptyset^n)=\sum_{k=1}^{n} t^{k-1}(t+t^{-1})^{-(n-1)}\,\tau_{k}\ldots\tau_{n-1}\left((t+t^{-1})^{-(N-n+1)}\,\Gamma_{w_\circ^{n-1}}(\zeta_\emptyset^n)\right)\,.
\end{equation}

By Remark \ref{symmi} and \eqref{topclass}, for any $i_1,\ldots,i_p$ between $1$ and $n-2$, we have
\[Y_{i_1}\ldots Y_{i_p}\,\zeta_\emptyset^n=[1,n]\ldots[n-1,n]\,Y_{i_1}\ldots Y_{i_p}\,\zeta_\emptyset^{n-1}\,.\]
Therefore, by induction and \eqref{defrho}, we have
\begin{align*}(t+t^{-1})^{-(N-n+1)}\,\Gamma_{w_\circ^{n-1}}(\zeta_\emptyset^n)&=[1,n]\ldots[n-1,n]\left((t+t^{-1})^{-(N-n+1)}\,\Gamma_{w_\circ^{n-1}}(\zeta_\emptyset^{n-1})\right)\\
&=[1,n]\ldots[n-1,n]\,\rho_1^{n-1}=\rho_{n}^{n}\,.
\end{align*}
Based on the above formula and Lemma \ref{lem1}, we can rewrite \eqref{eqxx} as follows:
\begin{align*}
(t+t^{-1})^{-N}\Gamma_{w_\circ^n}(\zeta_\emptyset^n)&=\sum_{k=1}^{n} t^{k-1}(t+t^{-1})^{-n+1}\,\tau_{k}\ldots\tau_{n-1}\,\rho_{n}^{n}\\
&=\sum_{k=1}^{n}\left(t^{k-1}(t+t^{-1})^{-(k-1)}\rho_k^{n}-t^k(t+t^{-1})^{-k}\rho_{k+1}^{n}\right)\\
&=\rho_1^{n}\,.
\end{align*}
But $\rho_1^{n}$ is identically $1$, as needed.
\end{proof}

\begin{proof}[Proof of Theorem {\rm \ref{mainthm} (2)} in type $A_{n-1}$] Recalling Example~\ref{sa1a2}, we see that we can assume $n\ge 4$. Fix $m$ between $2$ and $n$. Let $w_m$ be the (unique) highest representative of a coset in $W_n/W_{n-1}$ with $w_m(n)=m$; its length is $\ell(w_m)=N-m+1$. It is easy to check that the Schubert variety $X({w_m^{-1}})$ is non-singular by the well-known Lakshmibai-Sandhya pattern-avoidance criterion \cite{lasa}: the permutation $w_m^{-1}$ avoids the patterns $3412$ and $4231$.

We claim that $w=uv\le w_m$ (the factorization being the one in the above proof) if and only if $k:=w(n)=u(n)\ge m$. This can be seen using the well-known criterion for comparison in Bruhat order (see, e.g., \cite{BL}[Section~3.2]): $i_1\ldots i_n\le j_1\ldots j_n$ if and only if $\{i_1\ldots i_p\}\uparrow\le \{j_1\ldots j_p\}\uparrow$ for any $p=1,\ldots,n-1$, where the notation means that we arrange the elements of the two sets in increasing order, and we compare the two $p$-tuples entry by entry. Indeed, on the one hand it is not hard to see that, if $k\ge m$, then $w\le w_k\le w_m$. On the other hand, if $k<m$, then we cannot have $w\le w_m$ (just apply the above criterion with $p=n-1$). 

Based on the above facts, we can express $\gamma_{w_m}$ as in \eqref{gamma0}, but now we only sum over $\tau_k\ldots\tau_{n-1}$ in $W^{n-1}$ with $k\ge m$ (indeed, the fact that $X({w_m^{-1}})$ is non-singular implies that all the Kazhdan-Lusztig polynomials $P_{w_m,w}(t)$ with $w\le w_m$ are $1$, due to the well-known symmetry of these polynomials under inversion of their indices). Like in the above proof, and based on the previous result, we then calculate
\begin{align*}
{\mathfrak S}_{w_m^{-1}}=(t+t^{-1})^{-(N-m+1)}\Gamma_{w_m}(\zeta_\emptyset^n)&=\sum_{k=m}^{n} t^{k-m}(t+t^{-1})^{-(n-m)}\,\tau_{k}\ldots\tau_{n-1}\,\rho_{n}^{n}\\
&=\sum_{k=m}^{n}\left(t^{k-m}(t+t^{-1})^{-(k-m)}\rho_k^{n}-t^{k-m+1}(t+t^{-1})^{-(k-m+1)}\rho_{k+1}^{n}\right)\\
&=\rho_m^{n}\,.
\end{align*}
The proof is now concluded by applying Lemma~\ref{prodcl}.
\end{proof} 

\begin{proof}[Proof of Theorem {\rm \ref{mainthm} (2)} for $w=w_\circ$ in type $C_{n}$] We extend the notation in Example~\ref{bsc2} for type $C_2$. For start, the simple roots are $\alpha_0:=2\varepsilon_1$ and $\alpha_i:=\varepsilon_{i+1}-\varepsilon_i$. We let $[ij]:=y_{-(\varepsilon_j-\varepsilon_i)}$, for $i\ne j$ in $\{\pm 1,\,\ldots,\,\pm n\}$, where $\overline{\imath}:=-i$ and $\varepsilon_{\overline{\imath}}:=-\varepsilon_i$; in particular, $[\overline{\imath}i]:=y_{-2\varepsilon_i}$ and $[\overline{\imath}j]:=y_{-(\varepsilon_i+\varepsilon_j)}$. Also note that $[\overline{\jmath}\overline{\imath}]=[ij]$. 

It will be useful to represent an element (signed permutation) of the hyperoctahedral group $W_n$ of type $C_n$ as a bijection $w$ from $I:=\{\overline{n-1}<\ldots<\overline{1}<0<1<\ldots<n\}$ to $\{\pm 1,\,\ldots,\,\pm n\}$ with the property $i_{\overline{\jmath-1}}=\overline{\imath_j}$ for $j=1,\ldots,n$, where $i_k:=w(k)$ for $k\in I$. The action of the simple reflection $s_j$, for $j=0,\ldots,n-1$, consists in swapping the positions $j,j+1$ (and $\overline{\jmath},\overline{\jmath}+1$). This formalism has the advantage that it makes the definition \eqref{defxy} of the push-pull operator $Y_0$ completely similar to $Y_j$ for $j=1,\ldots,n-1$, while we can also define $Y_{\overline{\jmath}}$ similarly and we have $Y_{\overline{\jmath}}=Y_j$; indeed, it suffices to note that $[i_0 i_1]=[\overline{i_1}i_1]$, and that $[i_j i_{j+1}]=[i_{\overline{\jmath}} i_{\overline{\jmath}+1}]$ for $j=1,\ldots,n-1$. By defining $\tau_{\overline{\jmath}}$ via \eqref{taui}, we also have $\tau_{\overline{\jmath}}=\tau_j$.

We can define the functions $\rho_k^n:W_n\rightarrow S_n$ precisely like in \eqref{defrho}, except that $k$ is now any integer (usually between $\overline{n-1}$ and $n+1$) and $w=i_{\overline{n-1}}\ldots i_{\overline{1}}i_0i_1\ldots i_n$. Clearly $\rho_k^n$ is identically $0$ is $k>n$, and $\rho_{\overline{n-1}}^n$ is understood to be identically $1$, as usual. Observe that Lemmas~\ref{lem2} and \ref{lem1} still hold, where now $k\in\{0,\pm 1,\ldots,\pm(n-1)\}$. Note that, hidden in this formalism, are some peculiar applications of formula \eqref{eq1lem0} in Lemma~\ref{lem0}; for instance, the respective triple of roots $(\alpha,\beta,\alpha+\beta)$ can be one of the following, which contain a root $2\varepsilon_i$:
\begin{equation}\label{extriples}(2\varepsilon_i,\,\varepsilon_n-\varepsilon_i,\,\varepsilon_i+\varepsilon_n)\,,\;\;\;\;\mbox{or}\;\;\;\;(\varepsilon_i-\varepsilon_n,\,2\varepsilon_n,\,\varepsilon_i+\varepsilon_n)\,,\;\;\;\;\mbox{or}\;\;\;\;(\varepsilon_n-\varepsilon_i,\,\varepsilon_i+\varepsilon_n,\,2\varepsilon_n)\;\;\;\mbox{etc.}\end{equation}

The proof, as well as the notation, are completely analogous to the ones above for type $A_{n-1}$. Thus, we proceed by induction on $n$, with the base case having been treated in Example~\ref{sa1a2}. We now have $N:=\ell(w_\circ^n)=n^2$. The $2n$ lowest coset representatives in $W^{n-1}$ are of the form (in the window notation) $i_1\ldots i_{n}$ with $0<i_1<\ldots<i_{n-1}$, while $i_n$ can be positive or negative. As \eqref{klexp} still holds, we have
\begin{equation*}\gamma_{w_\circ^n}=\sum_{u\in W^{n-1}} t^{2n-1-\ell(u)}\tau_u\sum_{v\in W_{n-1}} t^{(N-2n+1)-\ell(v)}\tau_v=\sum_{u\in W^{n-1}}t^{2n-1-\ell(u)}\tau_u\gamma_{w_\circ^{n-1}}\,.\end{equation*}
Thus, we need to calculate
\begin{align*}
(t+t^{-1})^{-N}\Gamma_{w_\circ^n}(\zeta_\emptyset^n)&=\sum_{k=1}^{n} t^{n+k-1}(t+t^{-1})^{-(2n-1)}\,\tau_{k}\ldots\tau_{n-1}\left((t+t^{-1})^{-(N-2n+1)}\,\Gamma_{w_\circ^{n-1}}(\zeta_\emptyset^n)\right)\\
&+\sum_{k=1}^{n} t^{n-k}(t+t^{-1})^{-(2n-1)}\,\tau_{k-1}\ldots\tau_1\tau_0\tau_1\ldots\tau_{n-1}\left((t+t^{-1})^{-(N-2n+1)}\,\Gamma_{w_\circ^{n-1}}(\zeta_\emptyset^n)\right)\\
&=\sum_{k=-n+1}^{n} t^{n+k-1}(t+t^{-1})^{-(2n-1)}\,\tau_{k}\ldots\tau_{n-1}\left((t+t^{-1})^{-(N-2n+1)}\,\Gamma_{w_\circ^{n-1}}(\zeta_\emptyset^n)\right)\,.
\end{align*}
Like in type $A_{n-1}$, the bracket is calculated by induction as follows:
\[[\overline{n}n][\overline{n-1},n]\ldots[\overline{1}n][1n]\ldots[n-1,n]\rho_{\overline{n-2}}^{n-1}=\rho_{n}^{n}\,.\]
The proof is concluded in the same way as in type $A_{n-1}$, based on the new version of Lemma~\ref{lem1}.
\end{proof}

\begin{proof}[Proof of Theorem {\rm \ref{mainthm} (2)} in type $C_{n}$] Recalling Example~\ref{sa1a2}, we see that we can assume $n\ge 3$. Fix $m$ between $\overline{n-1}$ and $n$. Let $w_m$ be the (unique) highest representative of a coset in $W_n/W_{n-1}$ with $w_m(n)=m$; its length is
\[\ell(w_m)=\casetwo{N-(m+n-1)}{m>0}{N-(m+n)}{m<0}\]
It is easy to check that the Schubert variety $X({w_m^{-1}})$ is smooth by the pattern-avoidance criteria in \cite[Theorems~8.3.16~and~8.3.17]{BL}, see also the corresponding table in \cite[Chapter~13]{BL}. 

We then check that $w\le w_m$ if and only if $w(n)\ge m$. This can be seen in the same way as in type $A$, except that now we use Proctor's criterion for comparison in Bruhat order of signed permutations (see, e.g., \cite[Section~8.3]{BL}): $i_1\ldots i_n\le j_1\ldots j_n$ if and only if $\{i_p,\ldots i_n\}\uparrow\ge \{j_p\ldots j_n\}\uparrow$ for any $p=1,\ldots,n$, where the notation is the same as above, in type $A$. 

We conclude the proof by applying the same procedure as above to calculate ${\mathfrak S}_{w_m^{-1}}$; in fact, the obtained expression looks formally the same as the corresponding one above.
\end{proof} 

\begin{rem} The proof method we used above in types $A_{n-1}$ and $C_n$ does not work for type $B_n$. Indeed, when replacing the roots $2\varepsilon_i$ with $\varepsilon_i$, formula \eqref{eq1lem0} in Lemma~\ref{lem0} does not apply because the corresponding triples \eqref{extriples} are not of the form $(\alpha,\beta,\alpha+\beta)$. Our method does not apply to type $D_n$ either, for the following reason. Note first that we can adjust the notation used in type $C_n$; for instance, we let $\alpha_0:=\varepsilon_1+\varepsilon_2$, and we exclude the factor $[\overline{n}n]$ in the definition \eqref{defrho} of the functions $\rho_k^n$. However, the calculation of $Y_0\,\rho_1^n$ cannot be handled directly by formula \eqref{eq1lem0} in Lemma~\ref{lem0} because it involves expressions of the form
\[\frac{y_{\alpha+\beta}y_{\alpha+\gamma}}{y_\alpha}+\frac{y_\beta y_\gamma}{y_{-\alpha}}\,.\]
Instead, we can cancel the denominators by applying \eqref{eq1lem0} twice, but this produces a more complicated expression, with five terms instead of two. 
\end{rem}

\bibliographystyle{alpha}

\newcommand{\etalchar}[1]{$^{#1}$}

\end{document}